\documentclass[11pt,reqno]{amsart}
\usepackage{graphicx}
\usepackage{mathtools}

\usepackage{mathtools}
\mathtoolsset{showonlyrefs=true}
\usepackage{ulem}
\usepackage{color}

\usepackage{amsmath}
\usepackage{amssymb}
\usepackage{graphicx}
\usepackage{accents}
\usepackage{pstricks,pst-node,pst-coil,pst-plot}
\usepackage{mathtools}
\usepackage{ulem}
\usepackage{mathtools}
\usepackage{ulem}
\usepackage{graphicx}
\usepackage{color}
\usepackage{graphicx}
\usepackage{color}
\usepackage{dcolumn}
\usepackage{graphicx}
\usepackage{color}
\usepackage{dcolumn}
\usepackage{bm}
\usepackage{mathrsfs}
\usepackage{graphicx}
\usepackage{color}
\usepackage{dcolumn}
\usepackage{bm}
\usepackage{slashed}
\usepackage{epstopdf}
\usepackage{slashed}
\usepackage{bm}
\usepackage{slashed}
\usepackage{dcolumn}
\usepackage{bm}
\usepackage{slashed}
\usepackage{color}
\usepackage{color}
\def\m{\mathfrak{m}}

\newtheorem{thm}{Theorem}
\newcommand{\be}{\begin{equation}}
\newcommand{\ee}{\end{equation}}

 \numberwithin {equation}{section}

\newtheorem{prop}{Proposition}

\newcommand{\beq}{\begin{equation}}
\newcommand{\eeq}{\end{equation}}
\newcommand{\bee}{\begin{equation*}}
\newcommand{\eee}{\end{equation*}}
\def\d{\mathrm{d}}
\def\p{\partial}
\dedicatory{Dedicated to the memory of Professor Luen-Fai Tam}
\def\mbf{\mathfrak{m}_{_{\rm{B}}}}
\def\mb{\mathfrak{m}_{_{\rm{HB}}}}
\def\mby{\mathfrak{m}_{_{\rm{BY}}}}

\let\<\langle
\let\>\rangle

\newcommand{\definedas}{\mathrel{\raise.095ex\hbox{\rm :}\mkern-5.2mu=}}

\begin{document}
%%%%%%%%%%%%%%%%%%%%%%%%%%%%%%%%%%%%%%%%%%%%%%%%%%%%%%%%%%%%%%%%%

%%%%%%%%%%%%%%%%%%%%%%%%%%%%%%%%%%%%%%%%%%%%%%%%%%%%%%%%%%%%%%%%%
%\begin{center}
%  \framebox{\framebox{
 %     \vbox{This is project {\red (2+1)-D Hyperbolic Quasi-Local Mass}\\
 %      Current version {\blue note0210}, from {\blue note0209}, most
 %       recent changes by {\blue }.}  }}
%\end{center}
%%%%%%%%%%%%%%%%%%%%%%%%%%%%%%%%%%%%%%%%%%%%%%%%%%%%%%%%%%%%%%%%%

\title[The total geodesic curvature and the $(2+1)$-D hyperbolic mass]{The total geodesic curvature and the $(2+1)$-dimensional hyperbolic mass}

\author{Xiaokai He}
\address[Xiaokai He]{School of Mathematics and Statistics, Hunan First Normal University, Changsha 410205, China}
\email{sjyhexiaokai@hnfnu.edu.cn}
\author{Xiaoning Wu}
\address[Xiaoning Wu]{Institute of Mathematics, Academy of Mathematics and Systems Science and State Key Laboratory of Mathematical Sciences,
	Chinese Academy of Sciences, Beijing 100190, China}
\email{wuxn@amss.ac.cn}
\author{Naqing Xie}
\address[Naqing Xie]{School of Mathematical Sciences, Fudan
University, Shanghai 200433, China.}
\email{nqxie@fudan.edu.cn}
%\dedicatory{}

\begin{abstract}
In this paper, we derive a novel geometric inequality involving the total geodesic curvature, serving as an analogue of the Brown-York mass in the setting of 
$(2+1)$-dimensional gravity. The asymptotic behaviour of this newly defined quasi-local mass is examined for large ellipses in the time-symmetric slices of the Ba\~{n}ados-Teitelboim-Zanelli black hole solutions. Moreover, we establish an improved upper bound for the $(2+1)$-dimensional hyperbolic Bartnik mass.

\end{abstract}

%
% Latex changes the classification to 1991 ?!  it should be 2010!

%\date{}


\maketitle

%\tableofcontents

%%%%%%%%%%%%%%%%%%%%%%%%%%%%%%%%%%%%%%%%%%%%%%%%%%%%%%%%%%%%%%%%%%%%%%%%%%%%%%%%
%%%%%%%
\section{Introduction}
In the theory of general relativity with vanishing cosmological constant, an isolated gravitational system is modelled by a $(3+1)$-dimensional asymptotically flat spacetime. Defined at spatial infinity, the Arnowitt-Deser-Misner (ADM) mass \cite{ADM61} provides a measure of the total mass - or energy content - of such a system. As the simplest nontrivial example, we consider the Schwarzschild metric
\bee
\d s^2 =-\big(1-\frac{2m}{r}\big)\d t^2 + \frac{1}{1-\frac{2m}{r}}\d r^2 +r^2\d \theta^2 +r^2\sin^2\theta \d \varphi^2.
\eee
 On a constant-time slice, the induced metric approaches flatness as $r \rightarrow \infty$ and the ADM integral converges to the mass parameter $m$. This demonstrates that the ADM mass correctly captures the intuitive notion that a Schwarzschild black hole of mass $m$ has total energy $m$ as measured from infinity. But suppose we are interested only in a bounded region: how much mass does this region contain?  This motivates us to seek a notion of quasi-local mass quantities \cite{Sz09}. Among the existing geometrically natural definitions are the Bartnik mass \cite{Ba89} and the Brown-York mass \cite{BY93}, which we now introduce.

Let $ \Sigma$ be a $2$-sphere. Given a Riemannian metric $\gamma $ and a function $H$ on $\Sigma$, the Bartnik quasi-local mass of the triple $(\Sigma, \gamma, H)$ can be defined as
\bee
\mbf(\Sigma , \gamma ,H) = \inf \left \lbrace \m_{ADM} (M, g) \,\vert \,(M, g) \text{ is an admissible extension of }(\Sigma, \gamma,H)\right \rbrace.
\eee
Here $ \m_{ADM} (\cdot)$ is the ADM mass and
an asymptotically flat Riemannian $3$-manifold $(M, g)$ with boundary $\p M$ is  an {admissible extension of $(\Sigma, \gamma,H)$} if
\begin{itemize}
\item[i)]  $g$  has nonnegative scalar curvature;
\item[ii)]   $\p M$ with the induced metric is isometric to  $(\Sigma, \gamma)$ and, under the isometry, the mean curvature of $\p M$ in
$(M, g $) equals $H$; and
\item[ iii)]  $(M, g)$ satisfies certain non-degeneracy conditions that prevent $\m_{ADM} (M, g)$ from being arbitrarily small; for instance,
it is often required that $(M, g)$  contain no closed minimal surfaces (enclosing $\p M$), or that $\p M$ be outer-minimising in $(M, g)$.
\end{itemize}

In recent years there have been many results concerning estimates of the value of the Bartnik mass (cf. \cite{CCMP17,LS16,MS15,MP24,MWX20}). In the presence of a negative cosmological constant, the spacetime is assumed to be asymptotically anti-de Sitter. There is a natural analogue of the Bartnik mass for asymptotically hyperbolic (AH) manifolds with a negative lower bound on the scalar curvature \cite{CCM18}.  For a review and a more detailed discussion of the Bartnik mass, see e.g. \cite{Mc24}.

Suppose that the $2$-sphere $(\Sigma, \gamma)$ has positive Gauss curvature. By the work of Nirenberg \cite{Nirenberg} (and also  Pogorelov \cite{Pogorelov}), $(\Sigma, \gamma)$ is isometric to some strictly convex surface in the Euclidean space $ \mathbb{R}^3$. Let $\bar H$ be the mean curvature of such an isometric embedding. The Brown-York mass of $(\Sigma, \gamma, H)$ is defined to be
\bee
\mby(\Sigma,\gamma,H)=\frac{1}{8\pi}\int_\Sigma \big(\bar H- H \big)\d \gamma.
\eee
Assume further that $(\Sigma, \gamma)$ is the boundary of a Riemannian region  $(\Omega,g)$ with nonnegative scalar curvature such that the mean curvature equals the strictly positive $H$; then $\mby(\Sigma,\gamma,H)\geq 0$, with equality only if $(\Omega,g)$ is isometric to a domain of  $ \mathbb{R}^3$ \cite{ST-JDG02}. 

Our discussion of mass in general relativity has so far been rooted in the $(3+1)$-dimensional world we inhabit. However, the conceptual framework extends naturally to other dimensions. In $(2+1)$-dimensional asymptotically flat vacuum gravity \cite{St63}, the only solution is locally flat. But there is a well-known conical singularity solution, sometimes called the `$(2+1)$-D Schwarzschild' in an informal sense, given by:
\bee
\d s^2 = -(1-8\mu)\d t^2+ \frac{1}{1-8\mu}\d r^2 +r^2\d \varphi^2.
\eee
Here $\mu$ is a mass parameter and there is a conical deficit angle $8\pi\mu$. It describes a point particle rather than a black hole, and there is no analogue of the ADM mass in $(2+1)$ dimensions.

To obtain a true black hole in $(2+1)$ dimensions, we introduce a negative cosmological constant. This leads us to the Ba\~{n}ados-Teitelboim-Zanelli (BTZ) black hole solution \cite{BTZ92}, which shares many features with its $(3+1)$-dimensional Schwarzschild counterpart while remaining exact and analytically tractable.
We will be interested in manifolds whose metrics asymptote to the hyperbolic $2$-metric at large distances, which we will call asymptotically locally hyperbolic (ALH). The Hamiltonian mass of such spaces has been established in the literature \cite{CCQW25,KT79}.

In this paper, we establish a new geometric inequality for the total geodesic curvature, which is analogous to the Brown–York mass \cite{BY93} in the context of 
$(2+1)$-dimensional gravity. We then investigate the limiting behaviour of this new quasi-local mass for large ellipses within the time-symmetric slices of the BTZ black hole solutions. Furthermore, we obtain an improved upper bound estimate for the $(2+1)$-dimensional hyperbolic Bartnik mass.

\section{The total geodesic curvature and a new quasi-local mass}\label{Mt1}
\subsection{A geometric inequality for the total geodesic curvature}
Let $(\Omega,g)$ be a Jordan domain which is diffeomorphic to a two-dimensional closed disk and has boundary $\partial \Omega =\Sigma$. We choose an orthonormal frame $\{e_1,e_2\}$ so that $e_1$ is tangent to $\Sigma$ and $e_2$ is normal to $\Sigma$. The boundary curve $\Sigma$ is parameterised as
\be
[0,2\pi) \ni \varphi \mapsto i(\varphi) \in \Sigma. \ee
Denote by $L=2\pi r_0$ and $k=g(\nabla^g_{e_1}e_2,e_1)$ the length and the geodesic curvature of $\Sigma$ with respect to the outward normal $e_2$ of $\Sigma$, respectively.

If $g$ is the standard Euclidean metric $g^E$, then
\be
\int_\Sigma k \d s=2\pi.\ee
This is the well-known rotation index theorem for a plane curve \cite{doC}.

Now suppose that the metric $g$ has nonnegative Gauss curvature $K_g$. Is there any good estimate for the total geodesic curvature $\int_\Sigma k \d s$? It follows immediately from the Gauss-Bonnet formula \cite{doC},
\be
\int_\Sigma k \d s = 2\pi -\int_{\Omega} K_g \d g \leq 2\pi.\ee

This result can also be viewed from a conformal perspective. Let $f$ be the solution to the Dirichlet problem of the following Poisson equation:
\be
\begin{cases}
\Delta_g f  &= \ -K_g  \ \mbox{in} \  \Omega , \\
\ f  &= \ 0  \ \mbox{on} \  \Sigma ,
\end{cases}
\ee
where $\Delta_g$ and $K_g$ are the Laplacian and the Gauss curvature of the metric $g$. Then the conformal metric $\bar g=e^{2f}g$ has zero Gauss curvature and hence must be Euclidean.  Note that $K_g$ is assumed to be nonnegative; by the Hopf lemma, the curvature $\bar k=\bar g(\bar\nabla_{\bar {e}_1}\bar e_2,\bar e_1)$ with respect to the metric $\bar g$ is no less than that of the metric $g$. Therefore,
\be
\int_\Sigma k \d s \leq \int_\Sigma \bar k \d s=2\pi.\ee

It is natural to ask what we can say if the Gauss curvature has a negative lower bound. In this case the above technique encounters difficulties since the source term in the Poisson equation has no definite sign. However, the Gauss-Bonnet formula still gives an upper bound,
\be
\int_\Sigma k \d s = 2\pi +\int_{\Omega} \big(-K_g\big) \d g \leq 2\pi + |\Omega|_g\ee
where $|\Omega|_g$ is the area of $\Omega$ with respect to the metric $g$, assuming $K_g \ge -1$. This result concerns the effect of the boundary curve's curvature upon the internal geometry of the bounded domain.  Note that the upper bound involves the two-geometry of the entire $\Omega$.

Is it possible to give an upper bound for the total geodesic curvature that involves only the boundary geometry of $\Sigma$? The present work is of this nature. We will study the boundary behaviour of a compact 2-surface whose Gauss curvature is allowed to be negative.

A key ingredient in this work is the following elementary but useful observation: any simple closed curve in the physical space, equipped with its intrinsic Riemannian metric, can be  isometrically embedded into the standard hyperbolic 2-space as a round circle. This provides a canonical reference configuration against which geometric quantities on the original curve can be compared. More precisely, we have the following proposition.
\begin{prop}\label{prop:isometric-embedding}
Let $\Sigma $ be a simple closed curve in the physical 2-space $(\Omega, g)$ with arc-length parameter $s$ and let $L$ denote its total length. Set
$
r_0 := \frac{L}{2\pi}
$
and define an angular coordinate on $\Sigma$ by
\[
\varphi := \frac{s}{r_0} = \frac{2\pi s}{L}, \qquad 0 \leq s < L.
\]
Then the map
\[
\begin{aligned}
F: \Sigma &\longrightarrow \mathbb{H}^2\\
\varphi &\longmapsto \bigl( r_0,\, \varphi \bigr),
\end{aligned}
\]
where $\mathbb{H}^2$ is represented in the polar coordinates $(r,\varphi)$ with metric
\[
g_{-1} = \frac{\d r^2}{1+r^2} + r^2 \d\varphi^2,
\]
is an isometric embedding. Moreover, its image $F(\Sigma)$ is precisely the round circle in $\mathbb{H}^2$ centered at the origin with coordinate radius $r_0$.
\end{prop}

\begin{proof}
By the definition of the arc-length parameter, the induced metric on $\Sigma$ takes the form
\[
g|_{\Sigma} = \d s^2.
\]
Under the change of variables $\varphi = \frac{s}{r_0}$, we have $\d s = r_0  \d\varphi$, and since $s$ ranges over $[0,L)$, the new coordinate $\varphi$ ranges over $[0,2\pi)$. Hence the boundary metric becomes
\be\label{g1}
g|_{\Sigma} = r_0^2 \d \varphi^2. 
\ee

Now consider the image $F(\Sigma)$ in $\mathbb{H}^2$. By construction, the radial coordinate $r$ is constant along $F(\Sigma)$, namely $r = r_0$; consequently $\d r = 0$ on the image. The pullback of the hyperbolic metric $g_{-1}$ via $F$ is therefore
\be\label{g2}
F^\ast \bigl(g_{-1}\bigr)\big|_{F(\Sigma)}
= \left( \frac{\d r^2}{1+r^2} + r^2 \d\varphi^2 \right)\Big|_{r = r_0, \d r = 0}
= r_0^2 \d \varphi^2. 
\ee
Comparing \eqref{g1} and \eqref{g2}, we see that
\[
F^\ast\bigl(g_{-1}\bigr)\big|_{F(\Sigma)} = g|_{\Sigma},
\]
which proves that $F$ is an isometric embedding.

Finally, the image set is given explicitly by
\[
F(\Sigma) = \bigl\{ (r,\varphi) \in \mathbb{H}^2 | r = r_0, \varphi \in [0,2\pi) \bigr\},
\]
which is exactly a round circle in the hyperbolic 2-space $\mathbb{H}^2$ centered at the origin with coordinate radius $r_0$. 
\end{proof}

For a round circle $\Sigma_r$ with constant coordinate radius $r$ in the hyperbolic 2-space $\mathbb{H}^2$, its outward unit normal is given by $n^i=(\sqrt{1+r^2},0)$. With respect to the ambient metric $g_{-1}$, the geodesic curvature of $\Sigma_r$
can be computed explicitly as
\begin{equation}\label{tgk-1}
\hat k|_{\Sigma_r}= \frac{1}{\sqrt{\det(g_{-1})}}\, \partial_i \Bigl( \sqrt{\det (g_{-1})}\, n^i \Bigr)
= \frac{\sqrt{1+r^2}}{r}.
\end{equation}

With this preparation, we are now in a position to state the first main theorem of this paper.

\begin{thm}\label{T1}
A simple closed curve $\Sigma$ is parameterised as
\bee
[0,2\pi) \ni \varphi \mapsto i(\varphi) \in \Sigma. \eee
Denote by $L=2\pi r_0$ its length. One can isometrically embed the boundary curve $\Sigma$ as the round circle $F(\Sigma)=\{r=r_0\}$ into the standard hyperbolic 2-space $(\mathbb{H}^2,g_{-1}=\frac{\d r^2}{1+r^2}+r^2\d \varphi^2)$ by identifying the angle parameter $\varphi$. Denote by $\hat k$ the geodesic curvature of the embedded round circle $F(\Sigma)$ in $\mathbb{H}^2$. Given any positive function $k>0$ on $\Sigma$, there is a unique function $ u > 0 $ on $E =\mathbb{H}^2\setminus B_{r_0}$ such that the metric $ g_u = \frac{u^2 }{1+r^2}\d r^2 + r^2 \d \varphi^2 $
is asymptotically locally hyperbolic  and \\
\noindent (1) The Gauss curvature $K(g_u)$ of $g_u$ equals $-1$;\\
\noindent (2)
The geodesic curvature of $ \Sigma = \p E $ in $(E, g_u)$ equals $k$; and\\
\noindent (3) The Hamiltonian mass $H^0(g_u)$ satisfies
\be\label{MT}
 H^0(g_u)  \leq \frac{1}{\pi}\int_{F(\Sigma)} \big(\hat{k}-k \circ F^{-1}\big)\sqrt{1+r_0^2}\d s = \frac{1}{\pi}\int_\Sigma \big(\hat{k}-k\big)\sqrt{1+r_0^2}\d s,
\ee
where $\hat k$ is the geodesic curvature with respect to the metric $g_{-1}$.
\end{thm}

The purpose of this section is to prove this elegant geometric inequality \eqref{MT} by relating it to the  hyperbolic Hamiltonian mass \cite{CCQW25} of the extension $E$ which is determined by the  boundary geometry of $\Sigma$.

Motivated by the hyperbolic version of the mass expression in \cite[Line 8 before Section 2, Page 2358]{ST-CQG07}, we define a new quasi-local mass for the boundary curve $\Sigma$ as
\be\begin{split}
\m_{}(\Sigma):& =\frac{1}{\pi}\int_\Sigma \big(\hat{k}-k\big)\sqrt{1+r_0^2}\d s\\
&=\frac{1}{\pi}\int_{0}^{2\pi}\big(\hat{k}-k\big)r_0\sqrt{1+r_0^2}\d \varphi\\
&=2(1+r^2_0)-\frac{r_0\sqrt{1+r_0^2}}{\pi}\int_{0}^{2\pi}\big(k \circ F^{-1}\big)\d \varphi  \ \ \ \mbox{(by \eqref{tgk-1})}.
\end{split}
\ee
Here $k$ and $\hat k$ are the geodesic curvatures with respect to the physical metric $g$ and the reference metric $g_{-1}$, respectively. This new quasi-local mass $\m_{}(\Sigma)$ might be viewed as an analogue of the Brown-York mass \cite{BY93} in general relativity.

The idea of the proof of Theorem \ref{T1} here is as follows. We make use of the Bartnik-Shi-Tam extension \cite{Ba-JDG93,ST-JDG02} to glue the 2-surface $\Omega$ to the exterior region of the round circle $F(\Sigma)$ in $\mathbb{H}^2$ by identifying the two boundary curves and perturbing the hyperbolic metric in the transverse direction so that the scalar curvature remains $-2$ and the metric is asymptotically hyperbolic. This can be accomplished as in \cite{WY-CAG07} by solving a parabolic partial differential equation for the spherical foliation so that the geodesic curvatures on the
boundary of $\Omega$ and $E=\mathbb{H}^2\setminus B_{r_0}$ match along $\Sigma$, where $B_{r_0}$ is the disk in $\mathbb{H}^2$ enclosed by $F(\Sigma)$.  We show that there is a monotonicity formula for the difference of the integrals of the geodesic curvatures of the boundary curve in the physical space $\Omega$ and in the reference space $\mathbb{H}^2$. One concludes that the theorem is true.

\subsection{Prescribed scalar curvature}

The two-dimensional hyperbolic metric reads
\be
g_{-1}=\frac{\d r^2}{r^2+1}+r^2 \d \varphi^2.\ee

Let $u(r,\varphi)$ be a positive function. Define
\be
g^u=\frac{u^2 \d r^2}{r^2+1}+r^2 \d \varphi^2.
\ee
Direct calculation shows
\be
R(g^u)=-\frac{2}{u^3}\bigg(u+\frac{u^2u_{\varphi\varphi}}{r^2}-\frac{u_{r}(1+r^2)}{r}\bigg).
\ee
We require that $R(g^u)=-2$. Then $u(r,\varphi)$ satisfies the following parabolic equation:
\be\label{Para}
-u_r+\frac{u^2u_{\varphi\varphi}}{r(1+r^2)}=\frac{(u^3-u)r}{1+r^2}.
\ee

The solution of the parabolic equation \eqref{Para} can be compared to the solution of the ODE
\be\label{ODE}
\frac{\d f(r)}{\d r}=\frac{r(f-f^3)}{1+r^2},\ee
which can be solved explicitly.
Denote by
\be\label{f1}
f_1(r)=\Big(1+\frac{K_1}{1+r^2}\Big)^{-\frac{1}{2}}\ee
and
\be\label{f2}
f_2(r)=\Big(1+\frac{K_2}{1+r^2}\Big)^{-\frac{1}{2}}\ee
where

\be
K_1=\bigg[\frac{1}{(\max\limits_{\varphi \in S^1}u(r_0,\varphi))^2}-1\bigg](1+r_0^2),
\ee
and
\be
K_2=\bigg[\frac{1}{(\min\limits_{\varphi \in S^1}u(r_0,\varphi))^2}-1\bigg](1+r_0^2).
\ee

\begin{prop}
Let $r_0>0$ be a fixed constant. Assume that $u(r_0,\varphi)>0$. Then for any $R>r_0$, the parabolic equation \eqref{Para} has a positive solution on $[r_0,R]\times S^1$ satisfying
\be\label{ULS}
  \min\{1,f_2(r)\}\leq u(r,\varphi)\leq f_1(r).
\ee

\end{prop}
The proof is based on an application of the maximum principle. For any $\lambda>1$, we have
\be\label{Pf}
\begin{split}
\frac{\d \big(\lambda f_1\big)}{\d r} & = \frac{\lambda r(f_1-f_1^3)}{1+r^2}\\
& > \frac{ r(\lambda f_1-\lambda^3 f_1^3)}{1+r^2}.
\end{split}
\ee
It follows from \eqref{Para} and \eqref{Pf} that
\be\label{CI}
-\frac{\partial \big(u-\lambda f_1\big)}{\partial r} +\frac{u^2\big(u-\lambda f_1\big)_{\varphi\varphi}}{r(1+r^2)}> \frac{r}{1+r^2}\Big(-\big(u-\lambda f_1 \big) +\big(u^3-\lambda^3 f_1^3\big)  \Big).
\ee
Now suppose that $u>\lambda f_1$ somewhere. Note that initially $u(r_0,\varphi)<\lambda f_1(r_0)$; there exists a point $(\hat r,\hat \varphi)$ such that
$u(\hat r,\hat \varphi)=\lambda f_1(\hat r)$ and $u(r,\varphi)\leq \lambda f_1(r)$ for all $r\leq \hat r$.
Note that, at the point $(\hat r,\hat \varphi)$,
\be
-\frac{\partial \big(u-\lambda f_1\big)}{\partial r}|_{(\hat r,\hat \varphi)} +\frac{u^2\big(u-\lambda f_1\big)_{\varphi\varphi}}{r(1+r^2)}|_{(\hat r,\hat \varphi)} \leq 0.\ee
On the other hand,
\be
\frac{r}{1+r^2}\Big(-\big(u-\lambda f_1 \big) +\big(u^3-\lambda^3 f_1^3\big)  \Big)|_{(\hat r,\hat \varphi)}=0.
\ee
This leads to a contradiction with the strict inequality \eqref{CI} and shows that $u(r,\varphi) \leq \lambda f_1(r)$. Since $\lambda>1$ is arbitrary, we have
$u(r,\varphi) \leq  f_1(r)$. The lower bound inequality $\min\{1,f_2(r)\}\leq u(r,\varphi)$ can be obtained similarly. By taking the limit as $r\rightarrow \infty$, \eqref{f1}, \eqref{f2} and \eqref{ULS} yield that $r^2(u-1)$ converges to a smooth function $v_\infty(\varphi)$ on $S^1$.

\subsection{The mass expression}

We fix the background metric as
\be\label{back}
g_{-1}=\frac{\d r^2}{1+r^2}+r^2\d \varphi^2, \ \ \varphi \in [0,2\pi].\ee

A natural class of metrics associated with the background metric \eqref{back} is provided by the asymptotically locally hyperbolic (ALH) metrics \cite[Page 24, Eqn (4.59)]{CCQW25}, which can be rewritten in the form
\be
g=g_{-1}+\frac{\mu_{ij}\theta^i\theta^j}{r^2}+\mathcal{O}(r^{-3}),\ee
where $\mu_{ij}=\mu_{ij}(\varphi)$ is the mass aspect tensor and $\theta^1=r\d \varphi$, $\theta^2=\frac{1}{\sqrt{1+r^2}}\d r$.

The Hamiltonian mass \textit{\`{a} la} Kijowski-Tulczyjew reads \cite[Page 24, Eqn (4.60)]{CCQW25}
\be\label{H0}
H^0=\frac{1}{2\pi}\int_{\partial M}\big(\mu_{22}+2\mu_{11}\big)\d \varphi= \frac{1}{2\pi}\int_{\partial M}\mu \d \varphi,
\ee
where $\mu=\mu_{22}+2\mu_{11}$ is called the mass aspect function.

Note that the geodesic curvatures of $\Sigma_r$ with respect to $g^u$ and $g^{(u=1)}=g_{-1}$, denoted by $k_u$ and $k_{(u=1)}=k_{1}$ respectively, are related by
\be k_u=u^{-1}k_{1}. 
\ee It follows from \eqref{tgk-1} that $k_1=\frac{\sqrt{1+r^2}}{r}$. Set $u(r_0,\varphi)=\frac{\sqrt{r_0^2+1}}{r_0 k(\varphi)}$. Then the geodesic curvature of $F(\Sigma)$ in the metric $g^u$ equals the geodesic curvature of $\Sigma$ in $(\Omega, g)$ for the same value of $\varphi$.

Let $v=r^2(u-1)$ and recall the expansion of $g^u$:
\be
\begin{split}
g^u &=\big(1+\frac{v}{r^2}\big)^2(\frac{1}{1+r^2})\d r^2 +r^2\ d \varphi^2\\
&=\frac{\d r^2}{1+r^2}+r^2 \d \varphi^2+ \frac{2v_\infty}{r^2}\big(\frac{\d r}{\sqrt{1+r^2}}\big)^2 +\mathcal{O}(r^{-3}).
\end{split}
\ee
Thus,
\be
\begin{split}
H^0&=\frac{1}{2\pi}\int_\Sigma \big(2v_\infty\big)\d \varphi\\
&=\frac{1}{\pi}\int_\Sigma v_\infty\d \varphi.
\end{split}
\ee

\subsection{A monotonicity formula}
Suppose that $u(r,\varphi)$ is a positive function. Let
\be
\m(r)=\frac{1}{\pi}\int_{0}^{2\pi}\Big(r(k_{1}-k_u)\sqrt{1+r^2}\Big)\d \varphi.\ee
Then one has the following monotonicity formula:
\be\label{mono}
\begin{split}
\frac{\d \m(r)}{\d r}&=\frac{1}{\pi}\frac{\d }{\d r}\int_0^{2\pi}r(k_1-k_u)\sqrt{1+r^2}\d\varphi\\
&=\frac{1}{\pi}\int_0^{2\pi}\frac{\partial }{\partial r}\Big(r(k_1-k_u)\sqrt{1+r^2}\Big)\d\varphi\\
&=\frac{1}{\pi}\int_0^{2\pi}\big(-r(u+\frac{1}{u}-2)+u_{\varphi\varphi}\big)\d\varphi\\
& =-\frac{r}{\pi}\int_0^{2\pi}(u+\frac{1}{u}-2)\d\varphi\\
&\leq 0.
\end{split}
\ee

The following inequalities complete the proof of Theorem \ref{T1}.
\be
\begin{split}
\m(\Sigma)&=\frac{1}{\pi}\int_0^{2\pi}r_0\big(\hat k - k\big)\sqrt{1+r^2_0}\d \varphi\\
&=\frac{1}{\pi} \int_0^{2\pi}\Big(r(k_1-k_u)\sqrt{1+r^2}\Big)|_{r=r_0}\d\varphi\\
&\geq \lim_{r\rightarrow \infty}\frac{1}{\pi} \int_0^{2\pi}r(k_1-k_u)\sqrt{1+r^2}|_{r}\d\varphi \ \ \ \ \mbox{(by \eqref{mono})}\\
&=\lim_{r\rightarrow \infty}\frac{1}{\pi}\int_0^{2\pi}(1+r^2)(1-u^{-1})\d \varphi\\
&=\frac{1}{\pi}\int_0^{2\pi}v_\infty \d \varphi\\
&=H^0.
\end{split}
\ee

\section{Limits of the mass $\m(\Sigma)$ for large ellipses}

In this section, we study the limiting behaviour of the new quasi-local mass $\m(\Sigma)$ for ellipses in the Ba\~{n}ados-Teitelboim-Zanelli manifold. 

The metric we consider here is of the form 
\begin{equation}\label{BTZ}
g_m=\frac{1}{r^2-m}\d r^2+r^2\d\varphi^2
\end{equation}
where $m\in \mathbb{R}$ is a constant and the angle variable $\varphi \in [0,2\pi)$.

When $m>0$, by making the change of coordinate $r=\sqrt{m}\cosh\rho$, the metric \eqref{BTZ} can be rewritten as
\be
g_m=\d \rho^2+ m\cosh^2\rho \d \varphi^2.\ee
It can be defined on the region $(r,\varphi) \in [\sqrt{m},+\infty)\times [0,2\pi)$. In the physics literature, it is known as the time-symmetric slices of non-rotating BTZ black holes \cite{BTZ92} (also known as funnels in hyperbolic geometry). The boundary geodesic $r=\sqrt{m}$ is referred to as the horizon. The metric $g_m$ has constant scalar curvature $R(g_m)=-2$ and the Hamiltonian mass $H^0(g_m)$ equals $m+1$ \cite{CCQW25}.

When $m=0$, the associated surface with the metric \eqref{BTZ} is called the hyperbolic trumpet and $r=0$ is not included in the manifold.

For $m=-1$, the metric \eqref{BTZ} is actually the standard hyperbolic metric in polar coordinates $g_{-1}=\frac{\d r^2}{1+r^2}+r^2\d \varphi^2$ and the associated surface is just the standard hyperbolic 2-space $\mathbb{H}^2$.

However, when $m<0$ but $m\neq -1$, it is not globally isometric to the standard hyperbolic metric $g_{-1}$. By the change of coordinates $r=\sqrt{-m}\sinh \rho$ and $\varphi= \frac{\tilde \varphi}{\sqrt{-m}}$, the metric \eqref{BTZ} becomes $\d \rho^2+ \sinh^2\rho \d \tilde\varphi^2$. Thus, $\rho=0$ is a conical singularity with cone angle $2\pi\sqrt{-m}$.

The family of ellipses $\Sigma_R$ we consider is parameterised as
\be\label{ct}
r^2\cos^2\varphi + \frac{r^2\sin^2\varphi}{(1+\epsilon)^2}=R^2.
\ee

Note that the definition of $\m(\Sigma_R)$ relies only on the geometric data in a neighbourhood of $\Sigma_R$. The unit conormal vector of $\Sigma_R$ is
\begin{eqnarray}
n=f(r,\varphi)\bigg[\sqrt{\cos^2\varphi+\frac{\sin^2\varphi}{(1+\epsilon)^2}}\d r-
\frac{r\epsilon(\epsilon+2)\sin\varphi\cos\varphi}{(1+\epsilon)^2\sqrt{\cos^2\varphi+\frac{\sin^2\varphi}{(1+\epsilon)^2}}}\d\varphi\bigg],
\end{eqnarray}
where
\begin{eqnarray}
f(r,\varphi)=\bigg[(r^2-m)(\cos^2\varphi+\frac{\sin^2\varphi}{(1+\epsilon)^2})+\frac{\epsilon^2(2+\epsilon)^2\sin^2(2\varphi)}{4(1+\epsilon)^4(\cos^2\varphi+\frac{\sin^2\varphi}{(1+\epsilon)^2})}\bigg]^{-\frac{1}{2}}.
\end{eqnarray}
Direct calculation shows that the geodesic curvature of $\Sigma_R$ has the following asymptotic behaviour as $R$ goes to infinity:
\begin{eqnarray}
k=1-\frac{1}{R^2}k_2(\varphi)+\mathcal{O}(\frac{1}{R^3}),
\end{eqnarray}
where
\be
\begin{split}
k_2(\varphi)=&\frac{1}{8(1+\epsilon)^2\big(2+2\epsilon+\epsilon^2
+\epsilon(2+\epsilon)\cos2\varphi\big)}\bigg[8m+16m\epsilon\\
&+28\epsilon^2+20m\epsilon^3+28\epsilon^3+12m\epsilon^3+7\epsilon^4+3m\epsilon^4\\
&+4(2+m)\epsilon(4+6\epsilon+4\epsilon^2+\epsilon^3)\cos2\varphi\\
&+(1+m)\epsilon^2(2+\epsilon)^2\cos4\varphi\bigg].
\end{split}
\ee

Furthermore, the induced metric on $\Sigma_R$ is
\be
\begin{split}
\d l^2&=\bigg[\frac{r^2\epsilon^2(\epsilon+2)^2\sin^2(2\varphi)}{4(r^2-m)\big((1+\epsilon)^2\cos^2\varphi+\sin^2\varphi\big)^2}+r^2\bigg]\d\varphi^2
\end{split}
\ee
which yields
\be
\begin{split}
\d l&=\bigg[\frac{r^2\epsilon^2(\epsilon+2)^2\sin^2(2\varphi)}{4(r^2-m)\big((1+\epsilon)^2\cos^2\varphi+\sin^2\varphi\big)^2}+r^2\bigg]^{\frac{1}{2}}\d\varphi\\
&=\bigg[l_1R+\frac{l_2}{R}+\mathcal{O}(R^{-3})\bigg]\d\varphi,
\end{split}
\ee
where
\begin{eqnarray}
&&l_1=\sqrt{\frac{2(1+\epsilon)^2}{2+2\epsilon+\epsilon^2+\epsilon(2+\epsilon)\cos 2\varphi}},\\
&&l_2=\frac{\sqrt{2}\epsilon^2(2+\epsilon)^2\cos^2\varphi\sin^2\varphi}{
(1+\epsilon)\big(2+2\epsilon+\epsilon^2+\epsilon(2+\epsilon)\cos2\varphi\big)^{3/2}}.
\end{eqnarray}

Let $R_0(R)$ be
\be
\begin{split}
R_0(R):=&\frac{1}{2\pi}\int_{\Sigma_R}\d l\nonumber\\
=&\frac{R}{2\pi}\int_0^{2\pi}l_1\d\varphi+\frac{1}{2\pi R}\int_0^{2\pi}l_2d\varphi+\mathcal{O}(R^{-3}).
\end{split}
\ee
Then we have
\be
\begin{split}
\hat{k}&=\frac{\sqrt{1+R_0^2}}{R_0}=1+\frac{1}{2R^2c^2}+\mathcal{O}(R^{-3}).
\end{split}
\ee
where
\begin{eqnarray}
    c=\frac{1}{2\pi}\int_0^{2\pi}l_1\d\varphi.
\end{eqnarray}

Finally, the quasi-local mass  $\m(\Sigma_R)$  is
\be
\begin{split}
\m(\Sigma_R)&=\frac{1}{\pi}\int_{0}^{2\pi}\big(\hat{k}-k\big)R_0\sqrt{1+R_0^2}\d \varphi\\
&=\frac{1}{\pi}\int_0^{2\pi}\bigg[\frac{1}{2}+k_2(\varphi)c^2\bigg]\d\varphi+\mathcal{O}(R^{-2}).
\end{split}
\ee
As $R\rightarrow +\infty$, the limit of $\m(\Sigma_R)$ (depending on $m$ and $\epsilon$) becomes
\be
\begin{split}
\m_{\infty}(m,\epsilon):=&\lim_{R\rightarrow+\infty}\m(\Sigma_R)\\
=&\frac{1}{\pi}\int_0^{2\pi}\bigg[\frac{1}{2}+k_2(\varphi)c^2\bigg]\d\varphi.
\end{split}
\ee
For the coordinate round spheres, i.e., $\epsilon=0$, we have $k_2=\frac{m}{2},l_1=1,c=1$ which yields
\be
\begin{split}
\m_{\infty}(m,\epsilon=0)=m+1.
\end{split}
\ee

For $\epsilon=1$ and $m=1$, numerical calculation shows
\begin{eqnarray}
\m_{\infty}(m=1,\epsilon=1)=2.8848\neq 1+1.
\end{eqnarray}

For small $\epsilon$, we have
\begin{eqnarray}
&&l_1=1+\frac{1}{2}(1-\cos2\varphi)\epsilon+
\frac{3}{8}[\cos^2(2\varphi)-1]\epsilon^2+O(\epsilon^3),\\
&&k_2=\frac{m}{2}+\frac{1}{2}(4\cos2\varphi-M+M\cos2\varphi)\epsilon\nonumber\\
&&\ \ \ \ \ +\frac{1}{4}\bigg[7+4m-12\cos(2\varphi)-3m\cos(2\varphi)
-8\cos^2(2\varphi)\nonumber\\
&&\ \ \ \ \ -2m\cos^2(2\varphi)
+\cos(4\varphi)+m\cos(4\varphi)\bigg]\epsilon^2+O(\epsilon^3),
\end{eqnarray}
which yields
\begin{eqnarray}
c=1+\frac{\epsilon}{2}-\frac{3\epsilon^2}{16}+O(\epsilon^3),
\end{eqnarray}
and
\begin{eqnarray}
\m_{\infty}(m,\epsilon)=m+1+\frac{3(m+4)}{8}\epsilon^2+O(\epsilon^3).
\end{eqnarray}

In Fig.\ref{fig1}, we plot the graph of $\m_{\infty}(m,\epsilon)$ with respect to general $\epsilon$ for $m=1,0,-1$ respectively. It can be seen that the quasi-local mass $\m(\Sigma_R)$ may not approach the Hamiltonian mass $m+1$ if the ellipses are not coordinate round spheres.

\begin{figure}
\begin{tabular}{c}
\includegraphics[width=0.6\textwidth]{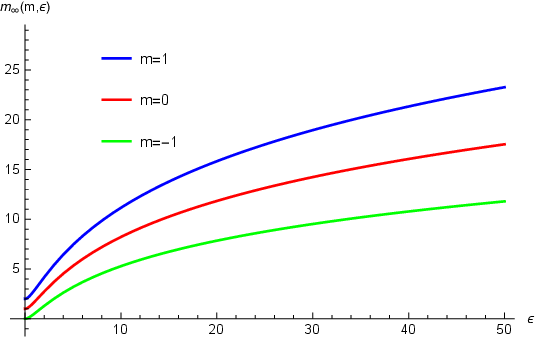}
\end{tabular}
\caption{Graph of $\m_{\infty}(m,\epsilon)$ with respect to $\epsilon$ for $m=1,0,-1$.}\label{fig1}
\end{figure}

\section{The $(2+1)$-dimensional hyperbolic Bartnik mass}
The Bartnik data is a triple $(\Sigma,r_0,k)$ consisting of a simple closed curve $\Sigma$ with length $2\pi r_0$ and a given function $k$ on $\Sigma$.
The  $(2+1)$-dimensional hyperbolic Bartnik quasi-local mass of the triple $(\Sigma, r_0, k)$ can be defined as follows.
\begin{equation*}
\begin{split}
&\ \ \ \mb(\Sigma , r_0 ,k)\\
&:= \inf \left \lbrace H^0(g) \,\vert \,(M, g) \text{ is an admissible hyperbolic extension of }(\Sigma, r_0, k)\right \rbrace.
\end{split}\end{equation*}
Here we say an asymptotically locally hyperbolic $2$-manifold $(M, g)$ with boundary $\p M$ is  an {admissible hyperbolic extension of $(\Sigma, r_0, k)$} if
\begin{itemize}
\item[i)] The Gauss curvature $K(g)$ of $g$ is no less than $-1$;
\item[ii)]   $\p M$ with the induced metric is isometric to  $\Sigma$ and, under the isometry, the geodesic curvature of $\p M$ in
$(M, g) $ equals $k$; and
\item[ iii)]  $(M, g)$ satisfies certain non-degeneracy conditions that prevent $H^0(g)$ from being arbitrarily small; for instance,
it is often required that $(M, g)$  contain no closed geodesics (enclosing $\p M$), or that $\p M$ be outer-minimising in $(M, g)$.
\end{itemize}
As in Section \ref{Mt1}, one can isometrically embed the boundary curve $\Sigma$ as the round circle $F(\Sigma)=\{r=r_0\}$ into the hyperbolic 2-space $\mathbb{H}^2$. Denote by $\hat k$ the geodesic curvature of the embedded round circle $F(\Sigma)$ in $\mathbb{H}^2$. Given any positive function $k>0$ on $\Sigma$, there is a unique function $ u > 0 $ on $M =\mathbb{H}^2\setminus F(\Sigma)$ such that the metric $ g_u = \frac{u^2 }{1+r^2}\d r^2 + r^2 \d \varphi^2 $ is asymptotically locally hyperbolic and
\begin{itemize}
\item[(1)] The Gauss curvature $K(g_u)$ of $g_u$ equals $-1$;
\item[(2)] The geodesic curvature of $ \Sigma = \p M $ in $(M, g_u)$ equals $k$; and
\item[(3)] The Hamiltonian mass $H^0(g_u)$ satisfies
$$
 H^0(g_u) \leq \frac{1}{\pi}\int_\Sigma \big(\hat{k}-k\big)\sqrt{1+r_0^2}\d s,
$$
where $\hat k$ is the geodesic curvature with respect to the metric $g_{-1}$.
\end{itemize}
Consequently, since $(M, g_u)$ is foliated by a $1$-parameter family of circles $\{ \Sigma_r \}_{r \geq r_0}$ with positive geodesic curvature,
it follows that
\be \label{oldb}
\begin{split}
 \mb(\Sigma, r_0, k) &\leq \frac{1}{\pi}\int_{\Sigma_{r_0}} \hat{k}\sqrt{1+r_0^2}\d s  -\frac{1}{\pi}\int_{\Sigma} k\sqrt{1+r_0^2} \d s\\
&=2(1+r_0^2) -\frac{1}{\pi}\int_{\Sigma} k\sqrt{1+r_0^2} \d s.
\end{split}
\ee
 regardless of the non-degeneracy condition iii) in the definition of $\mb(\Sigma, r_0,k)$.
 
In what follows, we apply the method in \cite{MX19} to obtain an improved upper bound estimate for the $(2+1)$-dimensional hyperbolic Bartnik mass.

\begin{thm}\label{MTV}
Let $ \Sigma$ be a simple closed curve with length $L=2\pi r_0$ and $k$ be a positive function on $\Sigma$ or identically zero $k\equiv 0$. There exists an asymptotically locally hyperbolic $2$-manifold $(M, g) $, diffeomorphic to $ \Sigma \times [r_0, \infty)$,
such that

\begin{itemize}
\item[i)] The Gauss curvature $K(g_u)$ of $g_u$ equals $-1$;
\item[ii)] $\p M$ with the induced metric is isometric to  $(\Sigma, r_0)$ and the geodesic curvature of $\p M$ in
$(M, g $) equals $k$ under the isometry;
\item[iii)]
$(M, g)$  is foliated by a $1$-parameter family of closed curves
with positive geodesic curvature;
and
\item [iv)] The Hamiltonian mass $H^0(g_u)$ satisfies
\bee\label{new}
 H^0(g_u) \leq 1+r_0^2.
\eee
Consequently,
\be\label{newb}
 \mb(\Sigma, r_0, k) \leq 1+r_0^2.
\ee
\end{itemize}
\end{thm}
In the case of $k\equiv 0$, the constructed extensions in this setting represent time-symmetric vacuum initial data with horizon boundary. Note that when the given function $k$ is sufficiently small, the estimate \eqref{newb} is much better than \eqref{oldb}.

Let us first prove Theorem \ref{MTV} for the case when $k$ is positive.
Let $m$ be a parameter lying in the interval $(-\infty,r_0^2)$ and let $N(r)=\sqrt{\frac{r^2-m}{1+r^2}}$. We consider the BTZ metric \cite{BTZ92}
\bee
g_m=\frac{1}{r^2-m}\d r^2+r^2\d\varphi^2.
\eee
\begin{prop}\label{P1}
The quantity
\be \label{Q}
\int_{ \Sigma_r }  (k_m - k_u) N(r) \sqrt{1+r^2} \d s
\ee
is non-increasing in $r$ on $[r_0,+\infty)$.
Here $k_m$ and $k_u$ are the geodesic curvatures of the curve $\Sigma_r=\{r=\mbox{const.}\}$ with respect to $g_m$ and $g_u$, respectively.
\end{prop}
\begin{proof}
Straightforward calculation shows that
\bee
k_m=\frac{\sqrt{r^2-m}}{r}
\eee
and
\bee
k_u=\frac{\sqrt{r^2+1}}{ur}.\eee
Recall that the function $u$ satisfies the following parabolic PDE \eqref{Para}
 \bee
-u_r+\frac{u^2u_{\varphi\varphi}}{r(1+r^2)}=\frac{(u^3-u)r}{1+r^2}.
\eee
Therefore,
\begin{equation*}
\begin{split}
&\ \ \ \frac{\d \Big(\int_{ \Sigma_r }  (k_m - k_u) N(r) \sqrt{1+r^2} \d s\Big)}{\d r}\\
&= \frac{\d \Big(\int_{0}^{2\pi}  (k_m - k_u) N(r) \sqrt{1+r^2}r\d \varphi \Big)}{\d r}\\
&= \int_{0}^{2\pi} \Big(  2r-\frac{r(2r^2+1-m)}{\sqrt{r^2-m}\sqrt{r^2+1}u} +\frac{\sqrt{r^2-m}}{\sqrt{r^2+1}}\frac{u_{\varphi\varphi}}{r} - \frac{\sqrt{r^2-m}}{\sqrt{r^2+1}u}(u^2-1)r        \Big)\d \varphi\\
&=-\int_0^{2\pi}\frac{r}{u \sqrt{r^2-m}\sqrt{r^2+1}}\Big(u \sqrt{r^2-m}-\sqrt{r^2+1}\Big)^2 \d \varphi\\
&\leq 0.
\end{split}
\end{equation*}
This completes the proof of Proposition \ref{P1}.
\end{proof}

\begin{prop}\label{P2} For any $r\in[r_0,+\infty)$, we have
\be\label{RS}
H^0(g_u) \leq m+1+ \frac{1}{\pi} \int_{\Sigma_{r}}\big(\frac{\sqrt{r^2-m}}{r}-k_u\big)\sqrt{1+r^2}N(r) \d s.
\ee
\end{prop}
\begin{proof}
One has
\begin{equation*}
\begin{split}
&\ \ \ \frac{1}{\pi}\int_0^{2\pi} (k_m - k_u) N(r) \sqrt{1+r^2}r\d \varphi\\
&= \frac{1}{\pi}\int_0^{2\pi}\Big( \big(\frac{\sqrt{r^2+1}}{r}\sqrt{\frac{r^2-m}{r^2+1}}-\frac{\sqrt{r^2+1}}{ur}\big) r \sqrt{r^2+1} \sqrt{\frac{r^2-m}{r^2+1}} \Big) \d \varphi\\
&= \frac{1}{\pi}\int_0^{2\pi}\Big( \big(\frac{\sqrt{r^2+1}}{r}\big(1-\frac{1+m}{r^2}+\mathcal{O}(r^{-4})\big)-\frac{\sqrt{r^2+1}}{ur}\big) \cdot\\
&\hspace{18mm}r \sqrt{r^2+1} \big(1-\frac{1+m}{r^2}+\mathcal{O}(r^{-4})\big) \Big) \d \varphi.
\end{split}
\end{equation*}
Then
\bee
\lim_{r\rightarrow \infty }\frac{1}{\pi}\int_0^{2\pi} (k_m - k_u) N(r) \sqrt{1+r^2}r\d \varphi =H^0(g_u)-(1+m).\eee
By the monotonicity of \eqref{Q}, Proposition \ref{P2} is proved.
\end{proof}

Note that $k_u(r,\varphi)$ is a positive function and independent of the parameter $m$. Taking $r=r_0$ in \eqref{RS} yields
\begin{equation}\label{P3}
\begin{split}
H^0(g_u) & \leq m+1+2(r_0^2-m)-2r_0\bar k_u(r_0)\sqrt{r_0^2-m} \\
&= \left( \sqrt{r_0^2 - m} - r_0\bar k_u(r_0) \right)^2 + r_0^2 + 1 - r_0^2\bar k^2_u(r_0),
\end{split}
\end{equation}
where
\bee
\bar{k}_u(r_0) =\frac{1}{2\pi}\int_0^{2\pi}k_u(r_0,\varphi) \d \varphi.
\eee
It is easy to see that by taking $m=r_0^2\big(1-\bar k_u^2(r_0)\big)$ the right-hand side of \eqref{P3} equals $r_0^2 + 1 - r_0^2\bar k^2_u(r_0)$.

Therefore,
\bee
H^0(g_u)\leq  1+r_0^2.
\eee

Now we consider the case when $k\equiv 0$. In this case, the construction is trivial. One simply identifies the curve $\Sigma$ as the boundary $\Sigma_{r_0}=\{r=r_0\} $ of the BTZ manifold with the metric
\bee
g_{r_0^2}=\frac{\d r^2}{r^2 -r_0^2} + r^2 \d \varphi^2.\eee
It is clear that the Gauss curvature $K(g_{r_0^2})$ of $g_{r_0^2}$ equals $-1$ and $H^0(g_{r_0^2})=1+r_0^2$. At first glance, it seems that the metric component $g_{rr}$ becomes degenerate at $r=r_0$. However, if we make the change of coordinate $r=r_0\cosh\rho$, the metric $g_{r_0^2}$ can be rewritten as
\bee
g_{r_0^2}=\d \rho^2+ r_0^2\cosh^2\rho \d \varphi^2.\eee

Finally, we have
\bee
\mb(\Sigma, r_0, k) \leq 1+r_0^2.\eee
Either when $k$ is positive or when $k\equiv 0$, the constructed extensions are foliated by a $1$-parameter family of closed curves $\Sigma_r=\{r=\mbox{const.}\}$ with positive geodesic curvatures. This completes the proof of Theorem \ref{MTV}.

\section*{Acknowledgments} The authors are grateful for the reviewer's suggestions.
 X. He is partially supported by  the National Natural Science Foundation of China (12475049). X. Wu was partially supported by the National Natural Science Foundation of China (12275350). N. Xie was partially sponsored by the Natural Science Foundation of Shanghai (24ZR1406000).

\end{document}